\RequirePackage{ifpdf} 
\ifpdf
\documentclass[11pt,pdftex]{article}
\else
\documentclass[11pt,dvips]{article}
\fi

\ifpdf
\usepackage{lmodern}
\fi
\usepackage[bookmarks, 
colorlinks=false, backref,plainpages=false]{hyperref}
\usepackage[T1]{fontenc}
\usepackage[latin1]{inputenc}
\usepackage[english]{babel}
\usepackage{graphicx}
\usepackage{url}
\usepackage{graphics}
\usepackage{epsfig,color}
\usepackage{amsmath,amssymb}
\usepackage{amsthm}
\renewcommand{\theequation}{\thesection.\arabic{equation}}

\newtheorem{thm}{Theorem}[section]
\newtheorem{cor}[thm]{Corollary}
\newtheorem{lem}[thm]{Lemma}

\newtheorem{rem}[thm]{Remark}

\begin{document}
\newcommand{\BX}{{\bf X}}
\newcommand{\cv}{{\cal V}}
\newcommand{\cW}{{\cal W}}
\newcommand{\co}{{\cal O}}

\renewcommand{\theequation}{\thesection.\arabic{equation}}
\def\@eqnnum{{\reset@font\rm (\theequation)}}

\def\abstract{
\advance \rightskip by 10mm
\advance \leftskip by 10mm
\vspace{-0.8em}
\noindent
\small{\bf Abstract.}
}
\def\endabstract{\par\normalsize\rm}

\def\Xint#1{\mathchoice
{\XXint\displaystyle\textstyle{#1}}%
{\XXint\textstyle\scriptstyle{#1}}%
{\XXint\scriptstyle\scriptscriptstyle{#1}}%
{\XXint\scriptscriptstyle\scriptscriptstyle{#1}}%
\!\int}
\def\XXint#1#2#3{{\setbox0=\hbox{$#1{#2#3}{\int}$}
\vcenter{\hbox{$#2#3$}}\kern-.5\wd0}}
\def\ddashint{\Xint=}
\def\dashint{\Xint-}

\def\a{\alpha}
\def\b{\beta}
\def\d{\delta}\def\D{\Delta}
\def\e{\epsilon}
\def\g{\gamma}\def\G{\Gamma}
\def\k{\kappa}
\def\lam{\lambda}\def\Lam{\Lambda}
\renewcommand\o{\omega}\renewcommand\O{\Omega}
\def\s{\sigma}\def\S{\Sigma}
\renewcommand\t{\theta}\def\vt{\vartheta}
\newcommand{\vphi}{\varphi}
\def\z{\zeta}

\newcommand{\tsigma}{\tilde{\s}}
\newcommand{\tbsigma}{\tilde{\bsigma}}
\def\te{\tilde{\e}}
\def\tu{\tilde{u}}

\newcommand{\bchi}{\mbox{\boldmath$\chi$}}
\newcommand{\bdelta}{\mbox{\boldmath$\delta$}}
\newcommand{\bepsilon}{\mbox{\boldmath$\epsilon$}}
\newcommand{\bfeta}{\mbox{\boldmath$\eta$}}
\newcommand{\bgamma}{\mbox{\boldmath$\gamma$}}
\newcommand{\bomega}{\mbox{\boldmath$\omega$}}
\newcommand{\bvphi}{\mbox{\boldmath$\varphi$}}
\newcommand{\bphi}{\mbox{\boldmath$\phi$}}
\newcommand{\bPhi}{\mbox{\boldmath$\Phi$}}
\newcommand{\bpsi}{\mbox{\boldmath$\psi$}}
\newcommand{\bPsi}{\mbox{\boldmath$\Psi$}}
\newcommand{\bsigma}{\mbox{\boldmath$\sigma$}}
\newcommand{\btau}{\mbox{\boldmath$\tau$}}
\newcommand{\bxi}{\mbox{\boldmath$\xi$}}
\newcommand{\brho}{\mbox{\boldmath$\rho$}}
\newcommand{\bbeta}{\mbox{\boldmath$\beta$}}
\newcommand{\bzeta}{\mbox{\boldmath$\zeta$}}

\def\bk{\boldsymbol{\kappa}}
\def\bmu{\boldsymbol\mu}
\def\bxi{\boldsymbol{\xi}}
\def\bz{\boldsymbol{\zeta}}

\def\ba{{\bf a}}
\def\bb{{\bf b}}
\def\bc{{\bf c}}
\def\be{{\bf e}}
\def\bff{{\bf f}}
\def\bg{{\bf g}}
\def\bn{{\bf n}}
\def\bp{{\bf p}}
\def\bq{{\bf q}}
\def\bs{{\bf s}}
\def\bt{{\bf t}}
\def\bu{{\bf u}}
\def\bv{{\bf v}}
\def\bw{{\bf w}}
\def\bx{{\bf x}}
\def\by{{\bf y}}
\def\bzz{{\bf z}}

\def\bD{{\bf D}}
\def\bE{{\bf E}}
\def\bF{{\bf F}}
\def\bH{{\bf H}}
\def\bJ{{\bf J}}
\def\bV{{\bf V}}
\def\bU{{\bf U}}
\def\bW{{\bf W}}
\def\bX{{\bf X}}
\def\bY{{\bf Y}}

\def\cA{{\cal A}}
\def\cC{{\cal C}}
\def\cD{{\cal D}}
\def\cE{{\cal E}}
\def\cF{{\cal F}}
\def\cG{{\cal G}}
\def\cI{{\cal I}}
\def\cJ{{\cal J}}
\def\cK{{\cal K}}
\def\cL{{\cal L}}
\def\cO{{\cal O}}
\def\cP{{\cal P}}
\def\cQ{{\cal Q}}
\def\cR{{\cal R}}
\def\cS{{\cal \Sigma}}
\def\cT{{\cal T}}
\def\cU{{\cal U}}
\def\cV{{\cal V}}

\def\scT{{_\cT}}
\def\sD{{_D}}
\def\sE{{_E}}
\def\sF{{_F}}
\def\sFz{{_{F_z}}}
\def\sK{{_K}}
\def\sI{{_I}}
\def\sb{{_b}}
\def\sN{{_N}}

\def\curl{{{\bf curl} \ }}
\def\rot{{\mbox{rot}\ }}
\def\BPI{{\bf \Pi}}

\def\cth{\cT_h}
\def\ctH{\cT_H}

\def\tJ{\tilde{\J}}

\def\hK{\widehat{K}}
\def\hx{\widehat{x}}
\def\hy{\widehat{y}}
\def\bhv{\widehat{\bv}}

\def\l{\ell}
\def\bl{\boldsymbol{\ell}}
\def\col{\colon}
\def\f12{\frac12}
\def\dfrac{\displaystyle\frac}
\def\dint{\displaystyle\int}
\def\nab{\nabla}
\def\p{\partial}
\def\sm{\setminus}
\def\dsum{\displaystyle\sum}
\newcommand{\pp}[2]{\frac{\partial {#1}}{\partial {#2}}}
\def\bzero{{\bf 0}}

\def\divv{\nab\cdot}
\def\divx{\nab_x\cdot}
\def\divtx{\nab_{t,x}\cdot}
\def\nabx{\nab_x}

\newcommand{\curlt}{{\nabla \times}}
\newcommand{\gperp}{\nabla^{\perp}}
\newcommand{\gradt}{\nabla\cdot}

\def\forallqq{\quad\forall\,}
\def\aph{A^{1/2}}
\def\amh{A^{-1/2}}

\def\osc{{\rm osc \, }}

\def\Im{{\rm Im}}
\newcommand{\tr}{{\rm tr}}
\def\divvr{{\rm div}}
\def\curllr{{\rm curl}}
\def\curll{{\rm curl}}
\def\curl{{\bf curl}}
\newcommand{\bgrad}{{\bf grad}}
\newcommand\diam{\mathrm{diam\,}}
\renewcommand\Im{\mathrm{Im\,}}
\def\Span{\mbox{Span}}
\def\supp{\mbox{supp\,}}
\newcommand{\trace}{{\rm trace}}

\newcommand{\tri}{|\!|\!|}
\newcommand{\ljump}{\lbrack\!\lbrack}
\newcommand{\rjump}{\rbrack\!\rbrack}
\newcommand{\bdm}{\begin{displaymath}}
\newcommand{\edm}{\end{displaymath}}
\newcommand{\beq}{\begin{equation}}
\newcommand{\eeq}{\end{equation}}
\newcommand{\beqa}{\begin{eqnarray}}
\newcommand{\eeqa}{\end{eqnarray}}
\newcommand{\beqas}{\begin{eqnarray*}}
\newcommand{\eeqas}{\end{eqnarray*}}
\newcommand{\ul}{\underline}
\newcommand{\wh}{\widehat}
\newcommand{\la}{\langle}
\newcommand{\ra}{\rangle}

\newcommand{\Lt}{L^2(\Omega)}
\newcommand{\Lts}{L^2(\Omega)^2}
\newcommand{\Ltc}{L^2(\Omega)^3}
\newcommand{\Ho}{H^1(\Omega)}
\newcommand{\Hoh}{H^1(\wh{\Omega})}
\newcommand{\Hoi}{H^1(\Omega_i)}
\newcommand{\Hos}{H^1(\Omega)^2}
\newcommand{\Hoc}{H^1(\Omega)^3}
\newcommand{\Hoch}{H^1(\wh{\Omega})^3}
\newcommand{\Hoci}{H^1(\Omega_i)^3}
\newcommand{\Hoz}{H^1_0(\Omega)}
\newcommand{\Ht}{H^2(\Omega)}
\newcommand{\Hti}{H^2(\Omega_i)}
\newcommand{\Hts}{H^2(\Omega)^2}
\newcommand{\Htc}{H^2(\Omega)^3}
\newcommand{\Htz}{H^0(\Omega)}
\newcommand{\Hh}{H^{1/2}(\Gamma)}
\newcommand{\Hhi}{H^{1/2}(\Gamma_i)}
\newcommand{\Hmh}{H^{-1/2}(\Gamma)}
\newcommand{\Hdiv}{H(\divvr;\,\Omega)}
\newcommand{\Hdivh}{H(\divv;\,\wh \Omega)}
\newcommand{\hcurl}{H(\curl\,A;\,\Omega)}
\newcommand{\Hcurl}{H(\curll\,A;\,\Omega)}
\newcommand{\Hcrl}{H(\curll\,;\,\Omega)}
\newcommand{\hcrl}{H(\curl\,;\,\Omega)}
\newcommand{\Hcrlh}{H(\curll\,;\,\wh\Omega)}
\newcommand{\hcrlh}{H(\curl\,;\,\wh\Omega)}
\newcommand{\Wdiv}{\BW_0(\mbox{\divv}\,;\,\Omega)}
\newcommand{\Wcurl}{\BW_0(\mbox{\curl}\,A;\,\Omega)}
\newcommand{\WcrossV}{\BW \times V}

\def\grad{\nabla}
\def\calS{{\cal S}}
\def\cH{{\cal H}}
\def\ba{{\mathbf{a}}}
\def\cN{{\cal N}}  

\def\bE{{\bf E}}
\def\bS{{\bf S}}
\def\br{{\bf r}}
\def\bW{{\bf W}}
\def\bLambda{{\bf \Lambda}}

\def\zT{{z_{_{\cT}}}}
\def\vT{{v_{_{\cT}}}}
\def\uT{{u_{_{\cT}}}}

\newcommand{\dd}{\underline{{\mathbf d}}}
\newcommand{\C}{\rm I\kern-.5emC}
\newcommand{\R}{\rm I\kern-.19emR}
\newcommand{\W}{{\mathbf W}}
\def\3bar{{|\hspace{-.02in}|\hspace{-.02in}|}}
\newcommand{\A}{{\mathcal A}}

\newcommand{\aA}{{ \a_{\sF,_A}}}

\newcommand{\aH}{{ \a_{\sF,_H}}}

\newcommand{\lJump}{[\![}
\newcommand{\rJump}{]\!]}
\newcommand{\jump}[1]{[\![ #1]\!]}

\newcommand{\red}[1]{{\color{red} {#1} }}

\def\Xint#1{\mathchoice
{\XXint\displaystyle\textstyle{#1}}%
{\XXint\textstyle\scriptstyle{#1}}%
{\XXint\scriptstyle\scriptscriptstyle{#1}}%
{\XXint\scriptscriptstyle\scriptscriptstyle{#1}}%
\!\int}
\def\XXint#1#2#3{{\setbox0=\hbox{$#1{#2#3}{\int}$}
\vcenter{\hbox{$#2#3$}}\kern-.5\wd0}}
\def\ddashint{\Xint=}
\def\dashint{\Xint-}

\title {Finite Element Methods for Interface Problems: \\
Robust and Local Optimal A Priori Error Estimates}
\author{Zhiqiang Cai\thanks{
Department of Mathematics, Purdue University, 150 N. University
Street, West Lafayette, IN 47907-2067, caiz@purdue.edu.
This work was supported in part by the National Science Foundation
under grants DMS-1217081 and DMS-1522707.}
\and Shun Zhang\thanks{Department of Mathematics,
City University of Hong Kong, Hong Kong SAR, China, shun.zhang@cityu.edu.hk.
This work was supported  in part by Research Grants Council of the Hong Kong SAR, China under the GRF Grant Project No. 11303914, CityU 9042090.}}
 \date{\today}
 \maketitle

\begin{abstract}
For elliptic interface problems in two- and three-dimension, 
this paper establishes a priori error estimates for 
the Crouzeix-Raviart nonconforming, the Raviart-Thomas mixed, and the discontinuous Galerkin
finite element approximations. These estimates are robust with respect to the 
diffusion coefficient and optimal with respect to local regularity of the solution. Moreover,
we obtain these estimates with no assumption on the distribution of the diffusion coefficient.
\end{abstract}

\section{Introduction}\label{intro}
\setcounter{equation}{0}

As a prototype of problems with interface singularities, this paper studies {\it a priori}
error estimates of various finite element methods for the following
interface problem (i.e., the diffusion problem with discontinuous coefficients):
 \begin{equation}\label{scalar}
 -\nabla\cdot \,(\a(x)\nabla\, u) = f
 \quad \mbox{in} \,\,\Omega
 \end{equation}
with homogeneous Dirichlet boundary conditions (for simplicity)
 \beq\label{bc1}
 u = 0 \quad \mbox{on } \p \O,
 \eeq
where  
$\Omega$ is a bounded polygonal domain in $\R^d$ with $d=2$ or $3$;
$f \in L^{2}(\O)$ is a given function;
and diffusion coefficient $\a(x)$ is positive 
and piecewise constant with possible
large jumps across subdomain boundaries (interfaces):
\[
 \a(x)=\a_i > 0\quad\mbox{in }\,\O_i
 \quad\mbox{for }\, i=1,\,...,\,n.
\]
Here, $\{\Omega_i\}_{i=1}^n$ is a partition of the domain $\O$ with
$\O_i$ being an open polygonal domain. The variational formulation for the interface problem in (\ref{scalar}) and
(\ref{bc1}) is to find $u\in H_{0}^1(\O)$ such that
 \beq\label{galerkin}
(\a \nabla u,\,\nabla v) = (f,v)\quad\forall\,\,v\in H_{0}^1(\O).
\eeq
It is well known that the solution $u$ of problem (\ref{galerkin}) belongs to $H^{1+s}(\O)$
with possibly very small $s> 0$.

Let $\cT=\{K\}$ be a regular triangulation of the domain
$\Omega$ (see, e.g., \cite{Cia:78, BrSc:08}). 
Denote by $h_K$ the diameter of the element $K$. Assume that
interfaces $\{\p\O_i\cap\p\O_j\,:\, i,j=1,\,...,\,n\}$ do not cut
through any element $K\in\cT$. For any element $K\in\cT$, 
denote by $P_k(K)$ the space of polynomials on $K$ with total degree less than or equal to $k$. Denote 
the continuous finite element space on the triangulation $\cT$ by
\[
 V_k^c= \{ v \in H^1_0(\O)\, :\,  v|_K \in P_k(K) \;\forall\, K\in\cT\}.
\]
Then the conforming finite element method is to 
find $u^c_k \in V^c_k$ such that
\beq \label{problem_c}
  (\a \nabla u^c_k,\, \nabla v) = (f,v)  \qquad \,\forall\, v\in
  V_k^c.
\eeq
The following a priori error estimate was established in \cite{BeVe:00}:
\beq \label{c:apriori}
 \|\a^{1/2}\grad(u-u^c_k)\|_{0,\O} 
 = \inf_{v \in V^c_k} \|\a^{1/2}\grad (u-v)\|_{0,\O}
 \leq C\, \left(\sum_{K\in \cT}h^{2s}\a_K\|\grad u\|^2_{s,K}\right)^{1/2}.
\eeq
Here and thereafter, 
we use C with or without subscripts to denote a generic positive constant that is
independent of the mesh parameter and the jump of $\a(x)$
but that may depend on the domain $\O$.
The estimate in (\ref{c:apriori})
is robust with respect to $\a$, but 
not optimal with respect to the local regularity since $s$ is a global exponent. 
This kind of a priori error estimate is not satisfactory. For example, for the well-known Kellogg's example of the
interface problem in \cite{Kel:74, CaZh:08a}, 
the solution of the underlying problem has low regularity on elements along the physical interfaces,
but is very smooth on elements away from the physical interfaces. 

By Sobolev's embedding theorem (see, e.g., \cite{Gri:85}), $H^{1+s}(\O)$, with 
$s>0$ for the two-dimension and $s>1/2 $ for the three-dimension, is embedded 
in $C^0(\O)$ and, hence, the nodal interpolation of the solution $u$
is well-defined. In \cite{DuSc:80}, it is proved that if $v \in H^{1+s}(K)$ with $s>0$ in the two-dimension,
then for $0<t\leq s$, the following estimate holds for the linear nodal interpolation $I_K$:
$$
\| v- I_K v\|_{0,K} \leq C h^{1+t} |\grad v|_{t,K}.
$$
With the same technique, we can also prove the result for $s>1/2 $ in the three-dimension.
This implies the following a priori error estimate that is not only 
robust with respect to the jump of $\a$ but also
with respect to the local regularity (see Section 3.3 of \cite{Pet:01} in the two-dimension).

\begin{cor} \label{optimal2d}
Let $u\in H^{1+s}(\O)$ with $s>0$ be the solution of problem {\em (\ref{c:apriori})}. 
Assume that $s >1/2$ for $d=3$ and that 
the restriction of $u$ on element $K$ belongs to $H ^{1+s_K}(K)$ for all $K\in\cT$. Then 
 \beq\label{apriori}
 \|\a^{1/2}\nabla (u - u^c_k)\|_{0,\O} 
  \leq   C \,\left(\sum_{K\in\cT} h_K^{2\min\{k,s_K\}}\a_K|\nabla u|^2_{s_K,K} 
  \right)^{1/2}.
 \eeq
 \end{cor}

\begin{rem}
In the case that $s\in (0,\,1/2]$ in the three-dimension, 
under the quasi-monotonicity assumption {\em (QMA)} on the distribution of the coefficient $\a(x)$
{\em (}see section~{\em 1.1)}, estimate {\em (\ref{apriori}) } 
may be obtained through comparison results with discontinuous Galerkin method or Cl\'{e}ment-type of 
interpolations  {\em(}see {\em \cite{CaZh:comp})}. 
\end{rem}

The a priori error estimate using local regularity in (\ref{apriori}) is 
the base for adaptive finite element 
methods to achieve equal discretization error distribution (see \cite{NoVe:12} for
examples in both the one- and two-dimensions) and, hence, is important. 
Moreover, the QMA is already restrictive enough
in the two-dimension and is much worse in the three-dimension. 
The purpose of this paper is to derive estimates of this type 
for the Crouzeix-Raviart nonconforming, 
the Raviart-Thomas mixed, and the discontinuous Galerkin finite element approximations.
These estimates hold 
when the solution of (\ref{galerkin}) has low global regularity, i.e., $s\in (0,\,1/2]$ in the
three-dimension, and, in particular, the distribution of the coefficient does not satisfy the QMA.
In this sense, these estimates are better than that of the conforming finite element methods.
Analysis for the mixed elements is rather straightforward. However, derivation of 
such estimates for the nonconforming and the discontinuous elements is non-trivial. 
In order to achieve them, we prove the robust C\'{e}a's Lemma type of results for the Crouzeix-Raviart nonconforming and
the discontinuous Galerkin finite element approximations for the first time.
Besides making use of both analytical approaches developed recently in the respective
\cite{CaYeZh:09} and \cite{Gud:10a}, we also need to establish 
new trace inequalities (see Lemmas~2.3 and 2.4). These trace inequalities
also play an important role in the a posteriori error estimates (see \cite{CaHeZh:15}).

Standard a priori error estimate for the discontinuous elements 
(see, e.g., \cite{ABCM:02, Ri:08}) requires the underlying problem being sufficiently smooth,
i.e, at least piecewise $H^{3/2+\epsilon}$, so that there is an error equation. 
For problems with low regularity, by carefully defining duality pairs on element interfaces,
in \cite{CaYeZh:09} we developed a non-standard variational formulation 
that, in term, leads to an error equation and then an a priori error estimate. 
The estimate in \cite{CaYeZh:09} is robust with respect to $\a$ without the QMA,
but not local optimal due to the use of a continuous approximation in our analysis.
An alternative approach was developed by Gudi \cite{Gud:10a} for the Poisson equation.
His approach compares the discontinuous solution with the continuous solution, 
and makes use of the efficiency bound of the a posteriori error estimation. Moreover,
it is applicable to problems with low regularity. Its application to interface 
problems with the Oswald analyzed introduced in \cite{CaYeZh:09} would yield
an a priori error estimate that is robust under the QMA. 

The paper is organized as follows. Section~2 introduces Sobolev spaces of fractional order
and establishes some new trace inequalities that play an important role in both the a priori
and a posteriori error estimates. Various finite element approximations are described 
in section~3. Robust and local optimal a priori error estimates without QMA are derived 
in section~4.

\subsection{Quasi-Monotonicity  Assumption}

To establish the a priori and, in particular, the a posteriori error estimates to be robust
with respect to the diffusion coefficient $\a(x)$, one often requires its distribution 
satisfying certain conditions. Hypothesis 2.7 in \cite{BeVe:00} is a monotonic condition
that is weaken to Quasi-Monotonicity Assumption (QMA) in \cite{Pet:02}. Such a condition
also appeared in the convergence analysis of the domain decomposition method 
in \cite{DrSaWi:96}.

\noindent{\bf Quasi-Monotonicity  Assumption.}
{\em Assume that any two different
subdomains $\bar{\O}_i$ and $\bar{\O}_j$, which share at least one
point, have a connected path passing from $\bar{\O}_i$ to
$\bar{\O}_j$ through adjacent subdomains such that the diffusion
coefficient $\a(x)$ is monotone along this path.}  

This assumption is needed in all previous papers on the robustness of the interface problem, e.g., \cite{BeVe:00, CaYeZh:09, CaZh:08a, CaZh:10a, CaZh:12a, Pet:02}. 
Robust estimates established in this paper do not require the QMA.

\section{Sobolev Space and Preliminaries}
\setcounter{equation}{0}

\subsection{Sobolev space of fractional order}

Let $\O$ be a non-empty open set in $\R^d$. 
We use the standard notation and definitions for the Sobolev
spaces $H^m(\O)^d$ and $H^m(\partial\O)^d$ with integer $m\ge 0$; the
standard associated inner products are denoted by $(\cdot , \,
\cdot)_{m,\O}$ and $(\cdot , \, \cdot)_{m,\partial\O}$, and
their respective norms (semi-norms) are denoted by $\|\cdot \|_{m,\O}$ and
$\|\cdot\|_{m,\partial\O}$
($|\cdot |_{m,\O}$ and
$|\cdot|_{m,\partial\O}$). We suppress the superscript $d$
because their dependence on dimension will be clear by context. We
also omit the subscript $\O$ from the inner product and norm
designation when there is no risk of confusion. For $m=0$,
$H^m(\O)^d$ coincides with $L^2(\O)^d$. In this case, the inner
product and norm will be denoted by $\|\cdot\|_0$ and
$(\cdot,\,\cdot)$, respectively. 

For $t\in (0,\,1)$, the following semi-norm 
\[
|v|_{t, \O} = \left( \int_{\O}\int_{\O} \dfrac{|v(x)-v(y)|^2}{|x-y|^{d+2t}} dxdy
\right)^{1/2}\quad 0<t<1
\]
is used to define Sobolev spaces of fractional order. For integer $m\geq 0$,
Sobolev space $H^s(\O)$ with $s= m+t$ is equipped with the norm
 \beq\label{sobolev}
 \|v\|_{s,\O} 
 =\left( \sum_{|\a| \leq m} \int_{\O} |\p^{\a}v|^2dx +|v|_{s,\O}^2\right)^{1/2},
\eeq
where $|v|_{s,\O}$ is a semi-norm defined by
\beq\label{sobolevsemi}
|v|_{s,\O} =\left( \sum_{|\a| = m}   |\p^{\a}v|_{t, \O}^2
\right)^{1/2}.
\eeq
Sobolev spaces with negative indecies are defined through duality.

Another way to define Sobolev spaces of fractional order is by the method of interpolation.
To this end, let $B_1 \subset B_0$ be Banach spaces. For $t>0$ and $u\in B_0$, 
define the $K$-functional by
$$
K(t,u) = \inf_{v\in B_1}(\|u-v\|^2_{B_0} + t^2 \|v\|_{B_1}^2 )^{1/2}.
$$
For $0<\t<1$, the interpolation space $B_\t = [B_0, B_1]_{\t}$ is a Banach space equipped with  the norm
\beq\label{internorm}
\|u\|_{[B_0,B_1]_{\t}} = N_\t \left(\int_0^{\infty} | t^{-\t} K(t,u)| ^2 \dfrac{dt}{t}\right)^{1/2},
\eeq
where $N_\t >0$ is a normalization factor.

For any real numbers $s_0 \leq s_1$, let $s=m+t=(1-\theta)s_0 + \theta s_1$
with $\t\in (0,\,1)$. It was shown (see Theorem B.8 in \cite{McL:00}) that 
 \[
 [H^{s_0}(\O), H^{s_1}(\O)]_\t = H^s(\O) 
 \]
and that the norms defined in (\ref{sobolev}) and (\ref{internorm}) 
are identical if the normalization factor is chosen to be
$N_\t = \sqrt{\dfrac{2 \sin{(\pi \t)}}{\pi}}$.
Moreover, for $v \in H^{s_1}(\O)$, it was shown (see Theorem B.1 in \cite{McL:00}) that 
\beq \label{normineq}
\|v\|_{s,\O} \leq \sqrt{\dfrac{\sin(\pi \theta)}
 { \pi\t (1-\t) }} \,\|v\|^{1-\theta}_{s_0,\O}\, \|v\|^{\theta}_{s_1,\O}.
\eeq

\begin{lem} \label{epsilon}
Let $s>0$, $t\in [0,\,s)$, and $K\in\cT$. Assume that $v$ is a given function in $H^{s}(K)$.
For any given $\epsilon >0$,
there exists a small $\delta\in (0,\,s-t)$, depending on $v$, such that
\beq \label{eps-bd}
\|v\|_{t+\delta,K} \leq (1+\epsilon) \,
 \|v\|_{t,K}. 
\eeq
\end{lem}

\begin{proof}
Obviously, (\ref{eps-bd}) holds for $v=0$. Assume that $v\not=0$.
For any $\delta\in (0,\,s-t)$, we have
$$
H^{t+\delta}(K) = [H^t(K), H^{s}(K)]_{\t} \quad \mbox{with  } \t = \dfrac{\delta}{s-t},
$$
which, together with (\ref{normineq}), implies
$$
\|v\|_{t+\delta,K}
 \leq \sqrt{\dfrac{\sin(\pi \theta)}{ \pi\t (1-\t) }} \,\|v\|_{t,K}^{1-\t}\, \|v\|_{s,K}^{\t} 
 = \sqrt{\dfrac{\sin(\pi \theta)}{ \pi\t (1-\t) }} \,
 \left(\dfrac{\|v\|_{s,K}}{\|v\|_{t,K}} \right)^{\t}\, \|v\|_{t,K}.
$$
Now, (\ref{eps-bd}) is a consequence of the fact that
 \[
 \lim_{\t\to 0} \,
 \sqrt{\dfrac{\sin(\pi \theta)}{ \pi\t (1-\t) }} \,
 \left(\dfrac{\|v\|_{s,K}}{\|v\|_{t,K}} \right)^{\t}
  =1.
\]
This completes the proof of the lemma.
\end{proof}

\begin{rem}
Since $\|v\|_{t,K}\leq \|v\|_{t+\delta,K}$, {\rm Lemma~\ref{epsilon}} implies that
 \[ 
 \lim\limits_{\delta \rightarrow 0^+}\|v\|_{t+\delta} = \|v\|_t. 
\]
Note that this continuity is not uniform with respect to $v$.
\end{rem}

\subsection{trace inequalities}

For any $K\in\cT$ and some $\alpha>0$, let 
 \[
 V^{1+\alpha}(K)=\{v\in H^{1+\a}(K)\,:\, \Delta\, v \in L^2(K)\}.
 \]

\begin{lem} 
Let $F$ be a face of $K\in\cT$ and let $s>0$.
Assume that $v$ is a given function in $V^{1+s}(K)$.
Then there exists a small $0<\delta <\min\{s,1/2\}$, depending on $v$, 
and a positive constant $C$ independent of $\delta$ such that
\beq \label{traceepsilon}
\|\nabla v\cdot \bn\|_{\delta-1/2,F}\leq
C\,\left(\|\nabla v\|_{0,K} + h_K\|\Delta v\|_{0,K}\right).
\eeq
\end{lem}

\begin{proof} For any $v\in V^{1+s}(K)$, 
it was shown in \cite{BeHe:01, CaYeZh:09} that for all $0<\delta <\min\{s,1/2\}$, we have
 \[
 \|\nabla v\cdot\bn\|_{\delta-1/2,F} 
 \leq C\,\left(\|\nabla v\|_{\delta,K}+h^{1-\delta}_K \|\Delta v\|_{0,K}\right),
 \]
which, together with Lemma~\ref{epsilon} with $t=0$ and the fact that $h^{-\delta}_K \leq 2$
for sufficiently small $\delta$, implies the validity of (\ref{traceepsilon}). 
This completes the proof of the lemma.
\end{proof}

\begin{lem} 
Let $F$ be a face of $K\in\cT$, $\bn_F$ the unit vector normal to $F$, and $s>0$.
Assume that $v$ is a given function in $V^{1+s}(K)$.
For any $w_h\in P_k(K)$, we have 
 \begin{eqnarray}\nonumber
 \int_F \left(\nabla v\cdot\bn_F\right)\, w_h \,ds 
  &\leq & C\, h_F^{-1/2}\|w_h\|_{0,F}
 \left(\|\nabla v\|_{0,K} + h_K\|\Delta v\|_{0,K}\right) \\[2mm] \label{tracecombined}
  & \leq  &  C\, h_K^{-1}\|w_h\|_{0,K}
 \left(\|\nabla v\|_{0,K} + h_K\|\Delta v\|_{0,K}\right).
\end{eqnarray}
\end{lem}

\begin{proof}
The second inequality in (\ref{tracecombined}) follows from the inverse inequality.
To show the validity of the first inequality in (\ref{tracecombined}), 
as discussed in \cite{CaYeZh:09}, 
$\int_F \left(\nabla v\cdot\bn_F\right)\, w_h \,ds$ may be viewed as a duality pairing 
between $H^{\delta-1/2}(F)$ and $H^{1/2-\delta}(F)$ for all $0<\delta<\min\{s,1/2\}$. 
It follows from the definition of the dual norm, the inverse inequality, and (\ref{traceepsilon}) for sufficiently small $\delta$ that
 \begin{eqnarray*}
 \int_F \left(\nabla v\cdot\bn_F\right)\, w_h \,ds 
  &\leq & \|\nabla v\cdot\bn\|_{\delta-1/2,F}\|w_h\|_{1/2-\delta,F}\\[2mm]
  & \leq & C\, h_K^{\delta-1/2}\|w_h\|_{0,F}
 \left(\|\nabla v\|_{0,K} + h_K\|\Delta v\|_{0,K}\right)\\[2mm]
  & \leq & C\, h_K^{-1/2}\|w_h\|_{0,F}
 \left(\|\nabla v\|_{0,K} + h_K\|\Delta v\|_{0,K}\right).
\end{eqnarray*}
This completes the proof of the first inequality in (\ref{tracecombined})
and, hence, the lemma.
\end{proof}

\begin{rem}
Generalizations of the above results to $\btau \in \{\btau \in H^{\a}(K)^s\,:\, \gradt \, \btau \in L^2(K)\} $ are obvious.
\end{rem}

\section{Various Finite Element Methods}
\setcounter{equation}{0}

Let $\cN$ be the set of vertices of the triangulation $\cT$ and $\cN_D$ 
be the collection of the vertices on the Dirichlet boundary.
Denote by $\cE_K$ the set of faces of element $K\in\cT$.
In this paper, face means edge/face in the two-/three-dimension.
Denote the set of all faces of the triangulation $\cT$ by
 \[
 \cE := \cE_{_I}\cup\cE_{_D},
 \]
where $\cE_{_I}$ and $\cE_{_D}$ are the respective sets of all interior and  boundary faces.  
For each $F \in \cE$, denote by  $h_F$  the diameter of the face $F$
and by $\bn_F$ a unit vector normal to $F$. 
For each interior face $F\in\cE_{_I}$, 
let $K_F^{-}$ and $K_F^{+}$ be the two elements sharing the
common face $F$ such that the unit outward normal vector of
$K_F^{-}$ coincides with $\bn_F$. When $F \in \cE_{\sD}$,
$\bn_F$ is the unit outward normal vector of $\p\O$ and denote the element by
$K_F^{-}$. For any $F\in\cE$, denote by $v|_F^{-}$ and $v|_F^{+}$,
respectively, the traces of a function $v$ over $F$.
Define jumps over faces by
$$
\jump{v}_F :=  \left\{
\begin{array}{lll}
 v|^{-}_F - v|_F^{+} \quad & F\in \cE_I,\\[2mm]
v|^{-}_F  \quad & F\in \cE_{_D}.
\end{array}
\right.
$$

Denote the continuous and discontinuous finite element spaces on the triangulation $\cT$ by
$$
V_k^c= \{ v \in H^1_0(\O)\, :\,  v|_K \in P_k(K) \;\forall\, K\in\cT\}
\,\mbox{ and }\,
D_k = \{ v \in L^2(\O) \,:\, v|_K \in P_k(K) \; \forall\, K\in\cT\},
$$
respectively. Denote the  Crouzeix-Raviart linear nonconforming finite element space by
$$
V^{cr} = \{ v \in L^2(\O) : v|_K \in P_1(K) \; \forall\, K\in\cT, \int_F \jump{v}ds =0, \; \forall F\in\cE\}
$$
and the $H(\divvr)$ conforming Raviart-Thomas finite element space by
$$
RT_k = \{ \btau \in H(\divvr;\O)\, :\, \btau|_K \in P_k(K)^d + \bx P_k(K) \; \forall\, K\in\cT\}.
$$

Denote by $\grad_h$ the discrete gradient operator that is defined element-wisely.
Then the nonconforming finite element method is to find
$u^{cr} \in V^{cr}$ such that 
\beq \label{problem_nc}
  ( \a\nabla_h u^{cr} ,\, \nabla_h v) = (f,v)  \quad\forall\,\, v\in
  V^{cr}.
\eeq

\subsection{mixed finite element method}

Introducing the flux 
 \[
 \bsigma = -\a(x)\nabla u, 
 \]
the mixed variational formulation for the problem in (\ref{scalar}) and
(\ref{bc1}) is to find
$(\bsigma,\,u)\in H(\divvr;\O)\times L^2(\O)$ such that
\begin{equation}\label{mixed}
 \left\{\begin{array}{lclll}
 (\a^{-1}\bsigma,\,\btau)-(\divv \btau,\, u)&=&0 \quad & \forall\,\, \btau \in
 H(\divvr;\O),\\[2mm]
 (\divv \bsigma, \,v) &=& (f,\,v)&\forall \,\, v\in L^2(\O).
\end{array}\right.
\end{equation}
Then the mixed finite element method is to find 
$(\bsigma^m_k,\,u^m_k) \in RT_k \times D_k$ such that
\begin{equation}\label{problem_mixed}
 \left\{\begin{array}{lclll}
 (\a^{-1}\bsigma^m_k,\,\btau)-(\divv \btau,\, u^m_k)&=&0
 \quad & \forall\,\, \btau \in RT_k,\\[2mm]
 (\divv \bsigma^m_k,\, v) &=& (f,\,v)&\forall \,\, v\in D_k.
\end{array}\right.
\end{equation}
Difference between (\ref{mixed}) and (\ref{problem_mixed}) yields 
the following error equation:
\begin{equation}\label{erroreq_mixed}
 \left\{\begin{array}{lclll}
 (\a^{-1}(\bsigma-\bsigma^m_k),\,\btau)-(\divv \btau,\, u-u^m_k)&=&0
 \quad & \forall\,\, \btau \in RT_k,\\[2mm]
 (\divv (\bsigma-\bsigma^m_k),\, v) &=& 0&\forall \,\, v\in D_k.
\end{array}\right.
\end{equation}

Let $f_{k}$ be the $L^2$ projection of $f$ onto $D_k$ for $k\geq 0$. 
Define local and global weighted oscillations by 
$$
\osc_\a(f,K) = \dfrac{h_K}{\sqrt{\a_K}}  \|f-f_{k-1}\|_{0,K} \quad\mbox{and}\quad
\osc_\a(f,\cT) = \left( \sum_{K\in\cT} \osc_\a(f,K)^2 \right)^{1/2},
$$ 
respectively.

\subsection{discontinuous Galerkin finite element method}

To describe disontinuous Galerkin finite element method, we need to introduce extra notations.
To this end, let $\o^+_F$ and $\o^-_F$ be weights defined on $F$ satisfying $w_F^+(x) + w_F^{-}(x) =1$,
and introduce the following weighted averages
\[
\{v(x)\}_w^F =  \left\{
\begin{array}{lll}
 w_F^- v^-_F+ w_F^+v^+_F  & F\in \cE_{I}, \\[2mm]
v|^{-}_F  \quad & F\in \cE_{_D}
\end{array}
\right. \mbox{ and }\, \{v(x)\}_F^w =  \left\{
\begin{array}{lll}
 w_F^+ v^-_F+ w_F^- v^+_F  & F \in \cE_{I}, \\[2mm]
0  \quad & F\in \cE_{_D}
\end{array}
\right.
\]
for all $F\in\cE$. 
A simple calculation leads to the
following identity:
 \begin{equation}\label{jump-id}
 \jump{ u v}_F = \{v\}^w_F\, \jump{u}_F + \{u\}_w^F\, \jump{v}_F.
 \eeq

For any $F\in \cE_I$, denote by $\a_F^+$ and  $\a_F^-$ the diffusion coefficients on
$K_F^+$ and $K_F^-$, respectively. 
Denote the arithmetic and the harmonic averages of $\a$ on $F \in \cE$ by
 \[
 \a_{\sF,_A} = \left\{\begin{array}{cl}
 \dfrac{\a_\sF^+  + \a_\sF^- }{2} &\quad F \in \cE_{\sI},\\[4mm]
 \a_\sF^-   &\quad F \in \cE_{\sD}
 \end{array}\right.
 \quad \mbox{and}\quad
 \a_{\sF,_H} = \left\{\begin{array}{cl}
  \dfrac{2\, \a_\sF^+  \a_\sF^- }{\a_\sF^+ + \a_\sF^-},&\quad F \in \cE_{\sI},\\[4mm]
 \a_\sF^-   &\quad F \in \cE_{\sD},
 \end{array}\right.
 \]
 respectively, which are equivalent to the respective maximum and minimum of $\a$:
\beq\label{a-h}
 \dfrac{1}{2}\max\{\a_\sF^+, \a_\sF^- \}
 \leq \a_{\sF,_A} 
 \leq \max\{\a_\sF^+, \a_\sF^- \}
 \, \mbox{ and } \, 
 \min\{\a_\sF^+, \a_\sF^- \}
 \leq \a_{\sF,_H} 
 \leq \dfrac{1}{2} \min\{\a_\sF^+, \a_\sF^- \} .
\eeq

For $s>0$, let 
 \begin{eqnarray*}
 H^{1+s}(\cT) &=& \{ v\in L^2(\O)\, :\, v|_K\in  H^{1+s}(K)\,\, \forall\, K\in \cT\}\\[2mm]
 \mbox{and }\,
 V^{1+s}(\cT) &=&\{ v \in H^{1+s}(K)\, :\, (\Delta v)|_K\in L^2(K)\,\, \forall\, K\in \cT \}.
 \end{eqnarray*}
In \cite{CaYeZh:09} we introduced the following variational formulation for the interface 
problem in (\ref{scalar}) and (\ref{bc1}): find $u\in V^{1+\epsilon}(\cT)$ with $\epsilon >0$
 such that
 \beq\label{DGV}
 a_{dg}(u,\,v) = (f,\,v)\quad\forall\,\, v \in V^{1+\epsilon}(\cT),
\eeq
where the bilinear form $a_{dg}(\cdot,\,\cdot)$ is given by
\begin{eqnarray*}
a_{dg}(u,v)&=&(\a\nabla_h u,\nabla_h v)
 +\sum_{F\in\cE}\int_F\gamma  \dfrac{\aH}{h_F} \jump{u}
 \jump{v}\,ds \\[2mm] \nonumber
&&\quad -\sum_{F\in\cE}\int_F\{\a\nabla
u\cdot\bn_F\}_{w}^F \jump{v}ds -
\sum_{F\in\cE}\int_F \{\a\nabla
v\cdot\bn_F\}_{w}^F\jump{u}ds.
\end{eqnarray*}
The $\gamma$ is a positive constant only depending on the shape of elements.
In order to guarantee robust error estimate with respect to $\a$, 
we choose the following harmonic weights:
 \beq\label{weights}
 w_F^\pm = \dfrac{\a_F^\mp}{\a_F^-+\a_F^+}.
 \eeq
The discontinuous Galerkin finite element method  is then to
seek $u^{dg}_k \in D_k$ such that
\beq\label{DGA}
a_{dg}(u^{dg}_k,\, v) = (f,\,v)\quad \forall\, v\in D_k.
\eeq
Difference between (\ref{DGV}) and (\ref{DGA}) leads to the following error equation
\beq \label{erroreq_dg}
a_{dg}(u-u^{dg}_k,\,v) = 0 \quad \forall \, v\in D_k.
\eeq
For simplicity, we consider only this symmetric version of the interior penalty discontinuous Galerkin finite element method 
since its extension to other versions of discontinuous Galerkin approximations is straightforward.
Define the jump semi-norm and the DG norm by 
$$
\|v\|_{J,F} = \sqrt{\dfrac{\aH}{h_F}} \|\jump{v}\|_{0,F}
\quad\mbox{and}\quad
\tri v\tri_{dg} = \left(\|\a^{1/2}\nabla_h v\|_{0,\O}^2+\sum_{F\in \cE}\|v\|_{J,F} ^2 \right)^{1/2},
$$
respectively, for all $v\in  H^1(\cT)$. It was shown in \cite{CaYeZh:09} that there exists
a positive constant $C$ independent of the jump of $\a$ such that
\beq \label{DGcoer}
C\, \tri v \tri_{dg}^2\leq a_{dg}(v,\, v) \quad \forall v \in D_k.
\eeq

\section{Robust and Local Optimal A Priori Error Estimates}
\setcounter{equation}{0}

\subsection{CR nonconforming finite element method}
Let 
 \begin{eqnarray*}
  W^{1,1}(\cT) &=&\{v\in L^2(\O)\, : \, v|_K\in W^{1,1}(K)\,\,\,\forall\, K\in \cT\}\\[2mm]
 \mbox{and}\quad 
 W (\cT) &=& \{v\in W^{1,1}(\cT)\, : \, \int_{F} \jump{v} ds =0\,\,\, \forall\, F\in \cE \}.
 \end{eqnarray*}
Denote by $\theta_F (\bx)$ the nodal basis function of $V^{cr}$ associated with the face
$F\in\cE$, i.e., 
 \[
 \dfrac{1}{|F^\prime|} \int_{F^\prime} \theta_F (\bx)\, ds
 =\delta_{FF^\prime}
 \,\,\,\forall\, F^\prime\in\cE,
 \]
where $\delta_{FF^\prime}$ is the Kronecker delta. The local and global Crouzeix-Raviart interpolants are defined respectively by
 \[
  I^{cr}_K v = \sum_{F\in \cE \cap \partial K} \left(\dfrac{1}{|F|} \int_{F} v ds \right) \theta_F(\bx)
 \quad\mbox{and}\quad
  I^{cr} v = \sum_{F\in \cE } \left(\dfrac{1}{|F|} \int_{F} v ds \right) \theta_F(\bx)
 \]
for the respective $v\in W^{1,1}(K)$ and $v\in W (\cT)$.
It was shown (see, e.g., Theorem 1.103 and Example 1.106 (ii) of \cite{ErGu:04}) that
for $v\in H^{1+t}(K)$ with $0\leq t  \leq 1$
\beq \label{localcr}
 \|v-I^{cr}_K v\|_{0,K} 
 \leq C \, h_{K}^{1+t}\, |\nabla v|_{t,K}. 
\eeq

\begin{thm}\label{apriori_nc}
Let $u$ be the soluion of {\em (\ref{galerkin})} and $u_K$ be its restriction on $K\in\cT$.
Assume that $u\in H^{1+s}(\O)\cap V^{1+s}(\cT)$ for some $s>0$ and 
that $u|_K\in H ^{1+s_K}(K)$ with element-wise defined $s_K>0$ for all $K\in\cT$.
Let $u^{cr}\in V^{cr}$ be the nonconforming finite element approximation in 
{\em (\ref{problem_nc})}. For both the two- and three-dimension, the following error estimates,
\begin{eqnarray} \nonumber
 \|\a^{1/2}\nabla_h(u-u^{cr})\|_{0} 
 &\leq&  C\, \left(\inf_{v\in V^{cr}}\|\a^{1/2}\nabla_h(u-v)\|_{0} + \osc_\a(f,\cT)\right)
 \\ [2mm] \label{ceacr}
 & \leq& C\, \left( \left( \sum_{K\in\cT} h_K^{2\min\{1,s_K\}}|\a^{1/2}\nabla u|^2_{s_K,K}
    \right)^{1/2} + \osc_\a(f,\cT) \right) 
\end{eqnarray}
hold, where $C$ is a positive constant independent of the jump of the diffusion
coefficient $\a$.
\end{thm}

\begin{proof}
The second inequality in (\ref{ceacr}) is an immediate consequence of (\ref{localcr}).
By Strang's lemma, to show the validity of the first inequality in  (\ref{ceacr}), it suffices to
prove
 \beq\label{4.9}
 \sup_{w\in V^{cr}}\dfrac{|(f,\,w)-(\a\nabla u,\,\nabla_h w)|}{\|a^{1/2}\nabla_h w\|_{0,\O}}
 \leq  C\, \left(\inf_{v\in V^{cr}}\|a^{1/2}\nabla_h(u-v)\|_{0,\O} + \osc_\a(f,\cT)\right).
\eeq

To this end, for any $w\in V^{cr}$ and any $F\in\cE$, by the fact that 
$\int_F \jump{w} ds =0$, the mean value of $w$ over $F$ is single-valued constant, i.e., 
 \[
 \bar{w}_{F} = \dfrac{1}{|F|}\,\int_F w|_{K^+_F}\,ds
   = \dfrac{1}{|F|}\,\int_F w|_{K^-_F}\,ds,
 \]
where $K^+_F$ and $K^-_F$ are two elements sharing the common face $F$. Moreover,
$\bar{w}_{F}=0$ for $F\in \cE_D$. 
Hence, by the continuity of the flux $\bn\cdot \a\nabla u$ across face $F\in \cE_I$, we have
 \beq\label{4.4}
 \sum_{K\in\cT}\sum_{F\in \p K}\int_F
 \big(\bn\cdot \a\nabla u\big)\, \bar{w}_{F}\, ds 
 = \sum_{F\in\cE}\int_{F} \jump{(\bn\cdot a\nabla u)\, \bar{w}_F}\, ds
 =0.
 \eeq
Now, it follows from (\ref{problem_nc}), integration by parts, (\ref{4.4}),
the fact that $\left(\bn_F\cdot\a \nabla v|_K\right)_F$ is a constant, and (\ref{tracecombined})
 that for all $v \in V^{cr}$
\begin{eqnarray*}
  (\a\nabla u,\,\nabla_h w)-(f,\,w)
   &=& \sum_{K\in\cT}\int_{\p K} (\bn\cdot a\nabla u)\, w\, ds
 =\sum_{K\in\cT}\sum_{F\in \p K}\int_F
  (\bn\cdot \a\nabla u)\, (w-\bar{w}_{F})\, ds \\[2mm]
  &=& \sum_{K\in\cT}\sum_{F\in \p K}\int_F
 \big(\bn\cdot \a\nabla (u-v)\big)\, (w-\bar{w}_{F})\, ds \\[2mm]
 &\leq & C\,\sum_{K\in\cT}\sum_{F\in \p K} 
 h^{-1/2}_K\|w-\bar{w}_{F}\|_{0,F}
 \left(\|\a\,\nabla (u-v)\|_{0,K} +  h_{K} \|f\|_{0,K}\right) 
  \\[2mm]
 &\leq & C\,\sum_{K\in\cT}\sum_{F\in \p K}
 \left(\|\a^{1/2}\nabla (u-v)\|_{0,K} +  h_{K}\a_K^{-1/2}\|f\|_{0,K}\right) 
  \,\|\a^{1/2}\nabla w\|_{0,K}.
\end{eqnarray*}
The last inequality is due to the fact that 
$ \|w-\bar{w}_{F}\|_{0,F} \leq C\, h_{K}^{1/2} \|\nabla w\|_{0,K}$. Now, the Cauchy-Schwarz inequality  gives
 \[
  \dfrac{\big| (\a\nabla u,\,\nabla_h w)-(f,\,w)\big|}
    {\|a^{1/2}\nabla_h w\|_{0,\O}}
 \leq  C\,\left(\inf_{v\in V^{cr}}\|\a^{1/2}\nabla_h(u-v)\|_0 
 + \left( \sum_{K\in\cT} h_{K}^{2}\a_{K}^{-1} \|f\|_{0,K}^{2}
    \right)^{1/2}  \right) 
 \]
for all $w\in V^{cr}$. Without QMA, in a similar fashion as the proof of the efficiency bound for the residual
error estimator of discontinuous Galerkin finite element method (Lemma 5.2 of \cite{CaYeZh:09}), we have
 \[
 h_K \a_K^{-1/2} \|f\|_{0,K} \leq C\, \left(\|\a^{1/2}\nabla_h (u-v)\|_{0,\o_K} + \osc_\a(f,K)\right)
 \] 
for all $v\in V^{cr}$ and all $K\in\cT$. Combining the above two inequalities
implies the validity of (\ref{4.9}). This completes the proof of
the theorem.
\end{proof}

Since linear conforming finite element solution $u^{c}_{1}\in V^{cr}$, we have
 \[
 \inf_{v\in V^{cr}}\|\a^{1/2}\nabla_h(u-v)\|_{0} 
 \leq \|\a^{1/2}\nabla_h(u-u^{c}_{1})\|_{0,\O},
\]
which, together with Theorem~\ref{apriori_nc}, implies the following robust
comparison result between linear conforming finite element and Crouzeix-Raviart nonconforming finite element approximations.

\begin{cor}\label{C-CRwithout}
Without QMA,  there exists a positive constant $C$ independent of the jump of the diffusion 
coefficient such that
\[
 \|\a^{1/2}\nabla_h(u-u^{cr})\|_{0,\O}
 \leq  C \left(\|\a^{1/2}\nabla_h (u-u^{c}_{1})\|_{0,\O}
+ \osc_\a(f,\cT)\right).
\]
\end{cor}

\subsection{RT mixed finite element method}
 
For fixed $s>0$, denote by $I^{rt}_k: \Hdiv \cap [H^s(\O)]^d \mapsto RT_k$ 
the standard $RT$ interpolation operator satisfying 
the following approximation property: for $\btau \in H^{s_\sK}(K)$
\beq \label{rti}
  \|\btau - I^{rt}_k \btau\|_{0,K}
  \leq C h_K^{\min\{s_K,k+1\}} |\btau|_{s_K,K} \quad\forall\,\, K\in \cT.
\eeq
(The estimate in (\ref{rti}) is standard for $s_\sK\geq 1$ and may be proved by 
the average Taylor series developed in \cite{DuSc:80} and the standard reference 
element technique with Piola transformation for $0<s_K<1$.) Denote by 
$Q_k:  L^2 (\O) \mapsto D_k$ the $L^2$-projection onto $D_k$.
The following commutativity property is well-known:
\beq\label{comm}
 \gradt (I^{rt}_k\,\btau)=Q_k\,\gradt\btau \qquad
 \forallqq\,\btau\in\Hdiv \cap H^s(\O)^d \,\mbox{ with }\, s>0.
\eeq

\begin{rem}
We use $\Hdiv \cap [H^s(\O)]^d$ instead of the choice 
$\{\btau\in L^p(\O)^d\mbox{  and  }\gradt \btau \in L^2(\O)\}$ for $p>2$ or $W^{1,t}(K)$ for $t>2d/(d+2)$ in  {\em \cite{RT:77,BrFo:91, BBF:13}} because this choice is more suitable for our analysis.
\end{rem}

\begin{thm}\label{apriori_mixed} 
Let $u$ and $(\bsigma^m_k,\,u^m_k)$ be the solutions of 
{\em (\ref{galerkin})} and {\em (\ref{problem_mixed})}, respectively.
Assume that $u\in H^{1+s}(\O)$ with $s>0$ and that $u|_K\in H ^{1+s_K}(K)$ with element-wise defined $s_K>0$
for all $K\in\cT$. Then
there exists a constant $C>0$ independent of the jump of $\a$ for both the
two- and three-dimension such that
\begin{equation}\label{err-bound-mixed}
\|\a^{-1/2}(\bsigma -\bsigma^{m}_k)\|_{0,\O} 
\leq \|\a^{-1/2}(\bsigma-I^{rt}_k \bsigma)\|_{0,\O} 
\leq C\left( \sum_{K\in\cT} h_K^{\min\{s_K,k+1\}} \a|\nabla u|_{s_K,K}\right)^{1/2}.
\end{equation}
\end{thm}

\begin{proof}
The second inequality in (\ref{err-bound-mixed}) is a direct consequence of the local
approximation property in (\ref{rti}). To establish the first inequality 
in (\ref{err-bound-mixed}), denote by 
 \[
  \bE = \bsigma -\bsigma^{m}_k
   \quad\mbox{and}\quad
   e = u- u^{m}_{k}
   \]
the respective errors of the flux and the solution.
The commutativity property in (\ref{comm}) and the second equations in 
(\ref{mixed}) and (\ref{problem_mixed}) lead to
 $$
 \gradt (I_k^{rt}\bsigma) = Q_k\,\gradt\bsigma = Q_k f = \divv
 \bsigma^{m}_k.
 $$
Now, it follows from the first equation in (\ref{erroreq_mixed}) and
the Cauchy-Schwarz inequality that
\begin{eqnarray*}
  \|\a^{-1/2}\bE\|_{0,\O}^2
  &= & (\a^{-1}\bE,\, \bsigma-I_k^{rt} \bsigma) +
  (\a^{-1}\bE,\, I_k^{rt} \bsigma-\bsigma^{m}_k)\\
  &=&(\a^{-1}\bE,\, \bsigma-I_k^{rt} \bsigma) +
  (\divv(I_k^{rt}  \bsigma-\bsigma^m_k),\,e)\\
  &=&(\a^{-1}\bE,\, \bsigma-I_k^{rt} \bsigma)
  \leq \|\a^{-1/2}\bE\|_{0,\O}\,\|\a^{-1/2}(\bsigma-I_k^{rt} \bsigma)\|_{0,\O},
\end{eqnarray*}
which implies the first inequality in (\ref{err-bound-mixed}). This completes the proof
of the theorem.
\end{proof}

\subsection{discontinuous Galerkin finite element method}

\begin{thm}\label{apriori_dg}
Let $u$ be the soluion of {\em (\ref{galerkin})} and $u|_K$ be its restriction on $K\in\cT$.
Assume that $u\in H^{1+s}(\O) \cap V^{1+s}(\cT)$ with $s>0$ and that $u|_K\in H ^{1+s_K}(K)$ with element-wise defined $s_K>0$
for all $K\in\cT$.
Let $u^{dg}_k\in D_k$ be the discontinuous
Galerkin finite element approximation in 
{\em (\ref{DGA})}. Let 
\beq
 app_\a(f,K) = \left\{\begin{array} {lll}
 \dfrac{h_K}{\sqrt{\a_K}} \|f-f_0\|_{0,K},  & \mbox{if  } 0<s_K <1, \\[4mm]
  h_K^ {\min \{k,s_K\} } \a_K^{1/2}|\nabla u|_{s_K,K}, & \mbox{if  } s_K \geq 1.
 \end{array}
 \right.
\eeq
In both the two- and three-dimension, we have the following error estimates:
\begin{eqnarray} \nonumber
 \tri u-u^{dg}_k\tri_{dg} &\leq& 
C\, \left(\inf_{v\in D_k}\tri u-v\tri_{dg} 
+ \osc_\a(f,\cT)\right)
\\ [2mm]  \label{ceadg}
& \leq& C 
 \left(\sum_{K\in\cT} h_K^{2s_K}|\a^{1/2}\nabla u|^2_{s_K,K} + app_\a(f,K)^2
    \right)^{1/2},
\end{eqnarray}
where $C$ is a positive constant independent of the jump of the diffusion
coefficient $\a$. 
\end{thm}

\begin{proof}
For any $F\in\cE$, it follows from the trace inequality and (\ref{a-h}) that for all $v\in D_k$
\begin{eqnarray*}
  \sqrt{\aH/h_F} \|\jump{u- v}\|_{0,F} 
 &\leq & \sqrt{\aH/h_F} \left( \|(u-v)|_{K^+_F}\|_{0,F} +  \|(u-v)|_{K^-_F}\|_{0,F} \right) \\[2mm]
 &\leq & C\, \sum_{\kappa=-,+}\left(
 h_{K^\kappa_F}^{-1}\|\a^{1/2}(u-v)\|_{0,K^\kappa_F} 
  + \|\a^{1/2}\nabla(u-v)\|_{0,K^\kappa_F} \right).
 \end{eqnarray*}
Since $u|_K \in H^{1+s_{_K}}(K)$ with $s_K\geq 1$, then
$f|_K =-\a \Delta u |_K \in H^{s_{_K}-1}(K)$.
It is easy to show that 
$$
\osc_\a(f,K)\leq 
  h_K^ {\min \{k,s_K\} } \a_K^{1/2}|\nabla u|_{s_K,K}.
$$ 
Now, the second inequality in (\ref{ceadg}) is a direct consequence of the first inequality in 
(\ref{ceadg}) and the elementwise approximation property of 
discontinuous piecewise polynomials. 
By the triangle inequality, we have
$$
 \tri u-u^{dg}_k\tri_{dg} 
 \leq \tri u-v\tri_{dg}  + \tri u^{dg}_k-v \tri_{dg} 
 \quad\forall\, v \in D_k.
$$

To show the validity of the first
inequality in (\ref{ceadg}), it suffices to prove that
 \beq\label{4.22}
 \tri u^{dg}_k-v \tri_{dg} 
 \leq C\, \left( \tri u-v\tri_{dg} 
+ \osc_\a(f,\cT)\right)
\quad\forall\, v \in D_k.
 \eeq
To this end, for any $v\in D_k$, let 
 \[
  e=u-v
  \quad\mbox{and}\quad
  e_k=u^{dg}_k-v.
  \]
It follows from the coercivity in (\ref{DGcoer}),
the error equation in (\ref{erroreq_dg}), the Cauchy-Schwarz inequality, the fact that
$\jump{u}_F=0$ for all $F\in\cE$, and
the first inequality in (\ref{tracecombined}) that  
 \begin{eqnarray*}
   && C\,\tri e_k \tri_{dg}^2
   \leq  a_{dg}(e_k,\, e_k)  =  a_{dg}(e,\, e_k) \\[2mm]
 &=&\!\! (\a \nabla_h e,\nabla_h e_k)
 +\!\!\sum_{F\in\cE} \int_F \dfrac{\gamma \,\aH \jump{e}\jump{e_k}}{h_F} ds 
  -\!\!\sum_{F\in\cE} \int_F  \{\a\nabla e_k\cdot\bn\}^F_w \jump{e} ds 
  -\!\!\sum_{F\in\cE} \int_F  \{\a\nabla e \cdot\bn\}^F_w \jump{e_k} ds
 \\[2mm]
 &\leq & C \,\Big\{\tri e \tri_{dg}\, \tri e_k\tri_{dg} 
 + \sum_{F\in\cE} \|\jump{e}\|_{0,F}   
    \|\{\a\nabla e_k \cdot\bn\}^F_w\|_{0,F} \\[2mm]
   && \quad +\sum_{F\in\cE} h_F^{-1/2}\|\jump{e_k}\|_{0,F}   
 \sum_{\kappa = -, +} 
 \ w^\kappa_F\a^\kappa_F\left(\|\nabla e\|_{0,K^\kappa_F} 
   + h_K \|\Delta e\|_{0,K^\kappa_F}\right)\Big\}\\[2mm]
 &\equiv & C\,\left(I_1+I_2+I_3\right).
\end{eqnarray*}

By the triangel, trace, and inverse inequalities, we have that 
 \[
 \|\{\a\nabla e_k\cdot\bn\}^F_w\|_{0,F}
 \leq C \sum_{\kappa = -, +}
  w^\kappa_\sF\,h^{-1/2}_{K^\kappa_F} \a^{1/2}_{K^\kappa_F}\,
  \| \a^{1/2}\nabla e_k\|_{0,K^\kappa_F}.
\]
With the choice of the weights in (\ref{weights}), a simple calculation shows that
  \[
  w^\kappa_\sF\sqrt{\dfrac{\a^\kappa_\sF}{\aH}}   
 \leq \dfrac{\sqrt{2}}{2}\,
 \quad\mbox{for }\, \kappa = -,\, +.
  \]
Together with the Cauchy-Schwarz inequality, we have
 \begin{eqnarray*}
 I_2
 &\leq & C\sum_{F\in\cE} \|\jump{e}\|_{J,F}  
    \sum_{\kappa = -, +} 
 \|\a^{1/2}\nabla e_k\|_{0,K^\kappa_F} 
 \leq C\, \tri e\tri_{dg}\,\tri e_k\tri_{dg} \\[2mm]
 I_3 &\leq & C\sum_{F\in\cE} \|\jump{e_k}\|_{J,F}  
    \sum_{\kappa = -, +} 
 \left(\|\a^{1/2}\nabla e\|_{0,K^\kappa_F} 
 + h_{K^\kappa_F} \|\a^{1/2}\Delta e\|_{0,K^\kappa_F}\right)\\[2mm]
 &\leq & C\, \tri e_k \tri_{dg}\,\left( \tri e \tri_{dg}
  +\left(\sum_{K\cT}h^2_K \a_K\|\Delta e\|_{0,K}^2\right)^{1/2}\right).
 \end{eqnarray*}
Combining those inequalities gives that
$$
  \tri e_k \tri_{dg}
 \leq C\, \left( \tri e \tri_{dg} 
 +\left(\sum_{K\in\cT} h_K^2\a_K \|\Delta e\|^2_{0,K} \right)^{1/2}\right).
$$
Now, (\ref{4.22}) is a direct consequence of the following efficiency bound
(see, e.g, Lemma 5.2 in \cite{CaYeZh:09} for the linear case):
$$
 h_K \a^{1/2}_K \|\Delta e\|_{0,K} 
 \leq C\, \left(\|\a^{1/2}\nabla e\|_{0,\o_K} + \osc_\a(f,\o_K)\right).
$$
This completes the proof of the inequality in (\ref{4.22}) and, hence, the theorem.
\end{proof}


\begin{thebibliography}{99}

 \bibitem{ABCM:02}
 {\sc D. N. Arnold, F. Brezzi, B. Cockburn, and L. D. Marini},
 {\em Unified analysis of discontinuous Galerkin methods for elliptic problems},
 SIAM J. Numer. Anal., 39 (2002), 1749-1779.

 \bibitem{BeHe:01} {\sc C. Bernardi and F. Hecht},
 {\em Error indicators for the mortar finite element discretization
 of the Laplace equation},  Math. Comp., 71:240 (2001), 1371-1403.

 \bibitem{BeVe:00}
 {\sc C. Bernardi and R. Verf\"urth},
 {\em Adaptive finite element methods for elliptic equations with
 non-smooth coefficients},
 Numer. Math., 85:4 (2000), 579-608.
 
 \bibitem{BBF:13}
 {\sc D. Boffi, F. Brezzi, and M. Fortin},
 {\em Mixed Finite Element Methods and Applications}, 
 Springer Series in Computational Mathematics, Vol. 44, Spinger, 2013.
 
  \bibitem{BrSc:08}
  {\sc 
  S. C. Brenner and L. R. Scott}, 
  {\em The Mathematical Theory of Finite Element Methods}, 3rd ed., Springer, 2008.

 \bibitem{BrFo:91}
 {\sc F. Brezzi and M. Fortin},
{\em Mixed and Hybrid Finite Element Methods}, Springer Series in Computational Mathematics, Vol. 15, Springer, 1991.

 \bibitem{CaYeZh:09}
 {\sc Z. Cai, X. Ye, and S. Zhang},
 {\em Discontinuous Galerkin finite element methods for interface problems:
 a priori and a posteriori error estimations},
SIAM J. Numer. Anal., 49:5 (2011), 1761--1787.

   \bibitem{CaZh:08a}
 {\sc Z. Cai and S. Zhang},
 {\em Recovery-based error estimator for interface problems:
 conforming linear elements},
 SIAM J. Numer. Anal., 47:3 (2009), 2132-2156.

 \bibitem{CaZh:10a}
 {\sc Z. Cai and S. Zhang},
 {\em Recovery-based error estimator for interface problems: mixed and nonconforming elements},
  SIAM J. Numer. Anal, 48:1 (2010), 30--52.
 
  \bibitem{CaZh:12}
 {\sc Z. Cai and S. Zhang},
 {\em Robust equilibrated residual error estimator for diffusion problems: conforming elements},
  SIAM J. Numer. Anal, 50 :1 (2012), 151-170. 
  
  \bibitem{CaZh:12a}
   {\sc Z. Cai and S. Zhang},
  {\em Robust residual- and recovery-based a posteriori error estimators for interface problems with flux jumps},
Numer. Methods for PDEs, 28:2 (2012), 476-491.
  
 
\bibitem{CaHeZh:15}
{\sc Z. Cai, C. He, and S. Zhang},
{\em Finite element methods for interface problems: Robust residual-based a posteriori error estimates}, arXiv:1510.06481 [math.NA], , 2015.
  
  \bibitem{CaZh:comp}
{\sc Z. Cai and S. Zhang},
{\em Finite element methods for interface problems: Robust comparison results}, manuscript, 
 2015.
  
\bibitem{Cia:78}
{\sc Ph. G. Ciarlet},
{\em The Finite Element Method for Elliptic Problems}, North-Holland, Amsterdam, 1978, Reprinted as "SIAM Classics in Applied Mathematics", No. 40, SIAM, Philadelphia, 2002.

\bibitem{DrSaWi:96}
{\sc M. Dryja, M. V. Sarkis, and O.B. Widlund},
{\em Multilevel schwartz method for elliptic problems with discontinuous in three dimensions},
Numer. Math., 72 (1996), 313-348.

  \bibitem{DuSc:80}
  {\sc T. Dupont and R. Scott},
{\em  Polynomial approximation of functions in Sobolev spaces}, 
Math. Comput., 34:150 (1980), 441-463.

\bibitem{ErGu:04}
{\sc A. Ern and J.-L. Guermond}, {\em Theory and Practice of Finite Elements}, vol. 159 of Applied Mathematical Series, Springer, New York, 2004.

 \bibitem{Gri:85}
 {\sc P. Grisvard},
 {\em Elliptic Problems in Nonsmooth Domains},
 Monographs and Studies in Mathematics, Vol. 24,
 Pitman, Boston, 1985.
 
 \bibitem{Gud:10a}
{\sc T. Gudi}, {\em A new error analysis for discontinuous 
finite element methods for linear elliptic problems}, Math. Comp. 79:272 (2010), 2169--2189.

  \bibitem{Kel:74}
 {\sc R.B. Kellogg},
 {\em On the Poisson equation with intersecting interfaces}, Appl. Anal., 4 (1975), 101-129.
 
\bibitem{McL:00}
{\sc W. McLean},
{\em Strongly Elliptic Systems and Boundary Integral Equations}, Cambridge University Press, 2000.

\bibitem{NoVe:12}
{\sc R. H. Nochetto and A. Veeser},
{\em Primer of adaptive finite element methods}, In {\em Multiscale and Adaptivity: Modeling, Numerics and Applications}, edited by
S. Bertoluzza, R. H. Nochetto, A. Quarteroni, K G. Siebert, and A. Veeser
Lecture Notes in Mathematics Volume 2040, Springer, 2012, pp 125-225

 \bibitem{Pet:01}
{\sc M. Petzoldt}, {\em
Regularity and error estimators for elliptic problems with discontinuous coefficients},
Ph.D Thesis, Freien Universitat Berlin, 2001.

  \bibitem{Pet:02}
 {\sc M. Petzoldt}, {\em A posteriori error estimators for elliptic
 equations with discontinuous coefficients},
 Adv. Comp. Math., 16:1 (2002), 47-75.
 
 \bibitem{RT:77}
{\sc P.A. Raviart and J.M. Thomas}, {\em A mixed finite element method for second order
elliptic problems}, in: Mathematical Aspects of the Finite Element Method (I.
Galligani and E. Magenes; eds.), Lecture Notes in Mathematics 606, Springer,
Berlin-Heidelberg-New York, 1977.

 \bibitem{Ri:08}
 {\sc B.  Rivi\`{e}re},
 {\em Discontinuous Galerkin Methods for Solving Elliptic and Parabolic Equations:
 Theory and Implementation},  SIAM, Philadelphia, 2008.


%
%

%
\end{thebibliography}
\end{document}